\documentclass[12pt]{amsart}

\usepackage[left=1in, right=1in, top=1in, bottom=1in]
{geometry}
\usepackage{amsfonts,amssymb,graphicx,amsmath,amsthm, float}
\usepackage{xypic}
\usepackage[all]{xy}
\theoremstyle{plain}
\newtheorem{theorem}{Theorem}[section]

\newtheorem{example}[theorem]{Example}

\newtheorem{corollary}[theorem]{Corollary}

\newcommand{\N}{\mathbb{N}}

\theoremstyle{remark}

\newcommand{\Z}{\mathbb{Z}}

\begin{document}

\title{$\tau$-Complete Factorization in Commutative Rings with Zero-Divisors}
        \date{\today}

\author{Christopher Park Mooney}
\address{Reinhart Center \\ Viterbo University \\ 900 Viterbo Drive \\ La Crosse, WI 54601}
\email{cpmooney@viterbo.edu}


\begin{abstract}
Much work has been done on generalized factorization techniques in integral domains, namely $\tau$-factorization.  There has also been substantial progress made in investigating factorization in commutative rings with zero-divisors.  There are many ways authors have decided to study factorization when zero-divisors present.  This paper focuses on the method $\tau$-complete factorizations developed by D.D. Anderson and A. Frazier.  There is a natural way to extend $\tau$-complete factorization to commutative rings with zero-divisors.  $\tau$-complete factorization is a natural way to think of refining factorizations into smaller pieces until one simply cannot refine any further.  We see that this notion translates well into the case of commutative rings with zero-divisors and there is an interesting relationship between the $\tau$-complete finite factorization properties and the original $\tau$-finite factorization properties in rings with zero-divisors developed by the author in 2012.
\\
\vspace{.1in}\noindent \textbf{2010 AMS Subject Classification:} 13A05, 13E99, 13F15
\end{abstract}
\maketitle
\section{Introduction}  
\indent There has been a substantial amount of research done on the factorization properties of commutative rings, especially domains.  Unique factorization domains (UFDs) are well understood and have been studied extensively over the years.  More recently, many authors have studied rings which satisfy various weakenings of the UFD conditions.  These factorization properties of domains have been extended in several ways to rings with zero-divisors.  Traditionally, in the domain case, authors have studied prime or irreducible factorizations. More recently, research has been done on generalizing the types of factorizations that have been studied to include things like co-maximal factorizations or using $\star$-operations to generalize factorization.  
\\
\indent Of particular interest to the current article is the 2011 work of D.D. Anderson and A. Frazier.  This is a survey article, \cite{ Frazier}, on the study of factorization in domains in which they introduce $\tau$-factorization.  The use of $\tau$-factorization yields a beautiful synthesis of many of these generalizations of factorizations studied in the integral domain case.  The goal then has been to extend this powerful approach of $\tau$-factorization to the case of a commutative ring with zero-divisors.  There have been several unique ways of studying factorization in commutative rings with zero-divisors, so this has led to many approaches to extending $\tau$-factorization.  
\\
\indent In \cite{Mooney}, the author used the methods established by D.D. Anderson and S. Valdes-Leon in \cite{Valdezleon} to extend many of the $\tau$-factorization definitions to work also in rings with zero-divisors.  In \cite{Mooney2}, the author investigated extending $\tau$-factorization using the notion of U-factorizations developed first by C.R. Fletcher in \cite{Fletcher, Fletcher2} and then studied extensively by M. Axtell, N. Baeth, and J. Stickles in \cite{Axtell, Axtell2}.   
\\
\indent In this paper, we investigate another particularly effective way to extend $\tau$-factorization to rings with zero-divisors.  This is the method of $\tau$-complete factorization.  This method was originally defined in the integral domain case in \cite{Frazier}.  In Section Two, we provide some necessary background definitions and theorems from the theory of $\tau$-factorization in domains as well as the theory of factorization in rings with zero-divisors.  In Section Three, we define what we refer to as $\tau$-complete factorizations in rings with zero-divisors.  These are $\tau$-factorizations in which the factorizations cannot be refined to create any strictly longer $\tau$-factorization.  We proceed to define several $\tau$-complete finite factorization properties rings may possess. In Section Four, we investigate the relationship between these new $\tau$-complete factorizations and the the previous $\tau$-irreducible factorizations studied in \cite{Mooney}.

\section{Preliminary Definitions and Results}
\indent For the purposes of this paper, we will assume $R$ is a commutative ring $1$.   Let $R^*=R-\{0\}$, let $U(R)$ be the set of units of $R$, and let $R^{\#}=R^*-U(R)$ be the non-zero, non-units of $R$.  As in \cite{Valdezleon}, we let $a \sim b$ if $(a)=(b)$, $a\approx b$ if there exists $\lambda \in U(R)$ such that $a=\lambda b$, and $a\cong b$ if (1) $a\sim b$ and (2) $a=b=0$ or if $a=rb$ for some $r\in R$ then $r\in U(R)$.  We say $a$ and $b$ are \emph{associates} (resp. \emph{strong associates, very strong associates}) if $a\sim b$ (resp. $a\approx b$, $a \cong b$).  As in \cite{Stickles}, a ring $R$ is said to be \emph{strongly associate} (resp. \emph{very strongly associate}) ring if for any $a,b \in R$, $a\sim b$ implies $a \approx b$ (resp. $a \cong b$).
\\
\indent Let $\tau$ be a relation on $R^{\#}$, that is, $\tau \subseteq R^{\#} \times R^{\#}$.  We will always assume further that $\tau$ is symmetric.  For non-units $a, a_i \in R$, and $\lambda \in U(R)$, $a=\lambda a_1 \cdots a_n$ is said to be a \emph{$\tau$-factorization} if $a_i \tau a_i$ for all $i\neq j$.  If $n=1$, then this is said to be a trivial \emph{$\tau$-factorization}.  
\\
\indent As in \cite{Mooney}, we say $\tau$ is \emph{multiplicative} (resp. \emph{divisive}) if for $a,b,c \in R^{\#}$ (resp. $a,b,b' \in R^{\#}$), $a\tau b$ and $a\tau c$ imply $a\tau bc$ (resp. $a\tau b$ and $b'\mid b$ imply $a \tau b'$).  We say $\tau$ is \emph{associate} (resp. \emph{strongly associate}, \emph{very strongly associate) preserving} if for $a,b,b'\in R^{\#}$ with $b\sim b'$ (resp. $b\approx b'$, $b\cong b'$) $a\tau b$ implies $a\tau b'$.  A \emph{$\tau$-refinement} of a $\tau$-factorization $\lambda a_1 \cdots a_n$ is a $\tau$-factorization of the form
$$(\lambda \lambda_1 \cdots \lambda_n)b_{11}\cdots b_{1m_1}\cdot b_{21}\cdots b_{2m_2} \cdots b_{n1} \cdots b_{nm_n}$$
where $a_i=\lambda_ib_{i_1}\cdots b_{i_{m_i}}$ is a $\tau$-factorization for each $i$.  We say that $\tau$ is \emph{refinable} if every $\tau$-refinement of a $\tau$-factorization is a $\tau$-factorization.  We say $\tau$ is \emph{combinable} if whenever $\lambda a_1 \cdots a_n$ is a $\tau$-factorization, then so is each $\lambda a_1 \cdots a_{i-1}(a_ia_{i+1})a_{i+2}\cdots a_n$.  It is easily checked that for a multiplicative $\tau$, it is combinable.  For a divisive $\tau$, it is refinable and associate preserving.
\\
\indent We now pause to supply the reader with a few examples of particularly useful or interesting $\tau$-relations to give an idea of the power of $\tau$-factorization.

\begin{example} Let $R$ be a commutative ring with $1$.
\begin{enumerate}

\item $\tau=R^{\#}\times R^{\#}$.  This yields the usual factorizations in $R$ and $\mid_{\tau}$ is the same as the usual divides.  $\tau$ is multiplicative and divisive and hence associate preserving, combinable and refinable.

\item $\tau=\emptyset$.  For every $a\in R^{\#}$, there is only the trivial factorization and $a\mid{_\tau} b \Leftrightarrow a=\lambda b$ for $\lambda \in U(R)$ $\Leftrightarrow a\approx b$.  Again $\tau$ is both multiplicative and divisive (vacuously).

\item Let $S$ be a nonempty subset of $R^{\#}$ and let $\tau=S\times S$, $a\tau b \Leftrightarrow a,b\in S$. So $\tau$ is multiplicative (resp. divisive) if and only if $S$ is multiplicatively closed (resp. closed under non-unit factors).  A non-trivial $\tau$-factorization is up to unit factors a factorization into elements from $S$.

\item Let $I$ be an ideal of $R$ and define $a \tau b$ if and only if $a-b\in I$.  This relation is certainly symmetric, but need not be multiplicative or divisive.  Let $R=\Z$ and $I=(5)$.  Consider $7\tau 2$ and $7 \tau 7$, but $7 \not \tau 14$, and $9 \tau 4$, but $2 \mid 4$ yet $9\not \tau 2$.

\item Let $a \tau b \Leftrightarrow (a,b)=R$, that is $a$ and $b$ are co-maximal.  These are the co-maximal factorizations studied by S. McAdam and R. Swan in \cite{Mcadam}.  This has been generalized in the following way.  Let $\star$ be a star-operation on $R$ and define $a\tau b \Leftrightarrow (a,b)^{\star}=R$, that is $a$ and $b$ are $\star$-coprime or $\star$-comaximal.  This particular operation has been studied more in depth by Jason Juett in \cite{Juettcomax}.

%\item Let $a\tau_n b \Leftrightarrow ab\neq 0$.  Notice in a domain, this is precisely $\tau = D^{\#} \times D^{\#}$, so this example is different from the usual factorization only for rings with zero-divisors.  $\tau$ is divisive, but not multiplicative.  Let $a\tau b$ and $a'\mid a$ and $b' \mid b$, say $a's=a$ and $b't=b$.  Then $a'sb't=ab\neq 0$, so certainly $a'b'\neq 0$ so $a'\tau b'$ as desired.  On the other hand, in $\Z/12\Z$ we have $2\tau 2$ and $2\tau 3$, but $2 \not \tau 6$. 

\item Let $a \tau_z b \Leftrightarrow ab=0$.  Then every $a\in R^{\#}$ is a $\tau$-atom.  The only nontrivial $\tau$-factorizations are $0=\lambda a_1 \cdot \ldots \cdot a_n$ where $a_i \cdot a_j=0$ for all $i \neq j$.  This example was studied extensively in \cite{Mooney} and is closely related to the zero-divisor graphs introduced by I. Beck in \cite{Beck}.  Zero-divisor graphs have since received a considerable amount of attention and have been studied and developed by many authors including, but not limited to D.D. Anderson, D.F. Anderson, M. Axtell, A. Frazier, J. Stickles, A. Lauve, P.S. Livingston, and M. Naseer in \cite{Axtellzdg, andersonzdg, davidanderson, Livingston}.  

%\item Let $a\tau b \Leftrightarrow a,b\in\text{Reg}(R)$.  Then this gives us the regular factorization studied in \cite{Valdezleon3}.  This is the inspiration for Section \ref{sec: regular}.

%\item Let $\tau \subseteq R^{\#}\times R^{\#}$, then we define $\tau_{reg}:=\tau \cap \left(Reg(R) \times Reg(R)\right)$.  Then this makes every zero-divisor into a $\tau_{reg}$-atom.  This is the type of factorization we would like to use to generalize the notion of $\tau$-factorizations to the regular factorizations studied in \cite{Valdezleon3}.  This will be studied more in depth in Section \ref{sec: regular}.
\end{enumerate}
\end{example}

\indent We now summarize several of the definitions of the following types of $\tau$-irreducible elements given in \cite{Mooney}, where other equivalent definitions can be found.  Let $a\in R$ be a non-unit.  Then $a$ is said to be \emph{$\tau$-irreducible} or \emph{$\tau$-atomic} if for any $\tau$-factorization $a=\lambda a_1 \cdots a_n$, $a\sim a_i$ for some $i$.  We will say $a$ is \emph{$\tau$-strongly irreducible} or \emph{$\tau$-strongly atomic} if for any $\tau$-factorization $a=\lambda a_1 \cdots a_n$, $a \approx a_i$ for some $a_i$.  We will say that $a$ is \emph{$\tau$-m-irreducible} or \emph{$\tau$-m-atomic} if for any $\tau$-factorization $a=\lambda a_1 \cdots a_n$, we have $a \sim a_i$ for all $i$.  We will say that $a$ is \emph{$\tau$-very strongly irreducible} or \emph{$\tau$-very strongly atomic} if $a\cong a$ and $a$ has no non-trivial $\tau$-factorizations.  
\\
\indent We introduce another type of $\tau$-irreducible element which will convenient in certain situations regarding complete factorizations.  We will say that $a$ is \emph{$\tau$-unrefinably irreducible} or \emph{$\tau$-unrefinably atomic} if $a$ has no non-trivial $\tau$-factorizations. 
\begin{theorem} Let $R$ be a commutative ring with $1$ and let $\tau$ be a symmetric relation on $R^{\#}$.  Let $a\in R$, be $\tau$-unrefinably irreducible.  Then any strong associate of $a$ is also $\tau$-unrefinably irreducible.
\end{theorem}
\begin{proof} Let $a\in R$ be $\tau$-unrefinably irreducible.  Let $a' \in R$ such that $a \approx a'$, say $a=\lambda a'$.  Suppose $a'$ were not $\tau$-unrefinably irreducible.  Then there is a non-trivial $\tau$-factorization of $a'$, say $a'= \mu b_1 \cdots b_n$ with $n \geq 2$.  But then $a=\lambda a'= (\lambda \mu) b_1\cdots b_n$ is a $\tau$-factorization with $n \geq 2$, contradicting the assumption that $a$ is $\tau$-unrefinably atomic.
\end{proof}
\begin{theorem}\label{thm: unrefinable} Let $R$ be a commutative ring with $1$ and $\tau$ be a symmetric relation on $R^{\#}$.  Let $a \in R$ be a non-unit.  The following diagram illustrates the relationship between the various types of $\tau$-irreducibles $a$ might satisfy where $\approx$ represents the implication requires a strongly associate ring:
$$
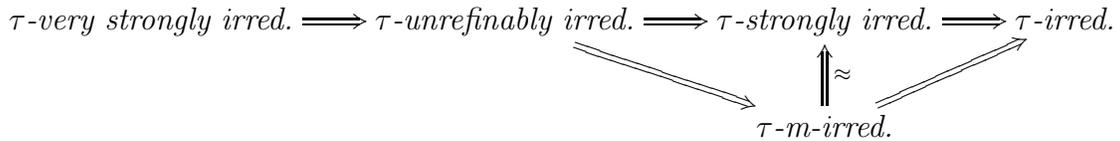
\xymatrix{
\tau\text{-very strongly irred.} \ar@{=>}[r] & \tau\text{-unrefinably irred.}\ar@{=>}[dr] \ar@{=>}[r] & \tau\text{-strongly irred.} \ar@{=>}[r]& \tau \text{-irred.}\\
& & \tau\text{-m-irred.}\ar@{=>}[u]_{\approx}\ar@{=>}[ur]   &}$$
\end{theorem}
\begin{proof} If $a$ is $\tau$-very strongly irreducible, then it is immediate that $a$ is also $\tau$-unrefinably irreducible.  We have simply removed that $a \cong a$ condition.  If $a$ is $\tau$-unrefinably irreducible, then the only $\tau$-factorizations of $a$ are of the form $a=\lambda b$ for some $b\in R$, but this shows $a \approx b$ and therefore $a \sim b$ proving $a$ is both $\tau$-m-atomic and $\tau$-strongly atomic.  The rest of the implications come from \cite[Theorem 3.9]{Mooney}.
\end{proof}
  %\indent From \cite[Theorem 3.9]{Mooney} and \cite{Mooneyregular}, we have the following relations where $\dagger$ represents the implication requires a strongly associate ring:
%$$\xymatrix{
%\tau\text{-very strongly irred.} \ar@{=>}[r] & \tau\text{-unrefinably irred.}\ar@{=>}[dr] \ar@{=>}[r] & \tau\text{-strongly irred.} \ar@{=>}[r]& \tau \text{-irred.}\\
%& & \tau\text{-m-irred.}\ar@{=>}[u]_{\dagger}\ar@{=>}[ur]   &}$$
 %\\
\indent This leads to the following definitions, found in \cite{Mooney}, for $\tau$-finite factorization properties a ring might possess.  While in \cite{Mooney}, the author did not use $\tau$-unrefinably atomic elements, we go ahead and add in the analogous definitions using this type of irreducible element as well. Given a commutative ring $R$ and a symmetric relation $\tau$ on $R^{\#}$, we let $\alpha \in \{$atomic, strongly atomic, m-atomic, unrefinably atomic, very strongly atomic$ \}$, $\beta \in \{$associate, strong associate, very strong associate$\}$.  Then $R$ is said to be \emph{$\tau$-$\alpha$} if every non-unit $a\in R$ has a $\tau$-factorization $a=\lambda a_1\cdots a_n$ with $a_i$ being $\tau$-$\alpha$ for all $1\leq i \leq n$.  We will call such a factorization a \emph{$\tau$-$\alpha$-factorization}.  We say $R$ satisfies \emph{$\tau$-ascending chain condition on principal ideals (ACCP)} if for every chain $(a_0) \subseteq (a_1) \subseteq \cdots \subseteq (a_i) \subseteq \cdots$ with $a_{i+1} \mid_{\tau} a_i$, there exists an $N\in \N$ such that $(a_i)=(a_N)$ for all $i>N$.
\\
\indent A ring $R$ is said to be a \emph{$\tau$-$\alpha$-$\beta$-unique factorization ring (UFR)} if (1) $R$ is $\tau$-$\alpha$ and (2) for every non-unit $a \in R$ any two $\tau$-$\alpha$ factorizations $a=\lambda_1 a_1 \cdots a_n = \lambda_2 b_1 \cdots b_m$ have $m=n$ and there is a rearrangement so that $a_i$ and $b_i$ are $\beta$.  A ring $R$ is said to be a \emph{$\tau$-$\alpha$-half factorization ring or half factorial ring (HFR)} if (1) $R$ is $\tau$-$\alpha$ and (2) for every non-unit $a \in R$ any two $\tau$-$\alpha$-factorizations have the same length.  A ring $R$ is said to be a \emph{$\tau$-bounded factorization ring (BFR)} if for every non-unit $a \in R$, there exists a natural number $N(a)$ such that for any $\tau$-factorization $a=\lambda a_1 \cdots a_n$, $n \leq N(a)$. A ring $R$ is said to be a \emph{$\tau$-$\beta$-finite factorization ring (FFR)} if for every non-unit $a \in R$ there are only a finite number of non-trivial $\tau$-factorizations up to rearrangement and $\beta$.  A ring $R$ is said to be a \emph{$\tau$-$\beta$-weak finite factorization ring (WFFR)} if for every non-unit $a \in R$, there are only finitely many $b\in R$ such that $b$ is a non-trivial $\tau$-divisor of $a$ up to $\beta$.  A ring $R$ is said to be a \emph{$\tau$-$\alpha$-$\beta$-divisor finite (df)} if for every non-unit $a \in R$, there are only finitely many $\tau$-$\alpha$ $\tau$-divisors of $a$ up to $\beta$.
\\
\indent In \cite[Theorem 4.1]{Mooney}, the author shows the following following diagram ($\nabla$ represents $\tau$ being refinable and associate preserving) holds when $\alpha \neq$ $\tau$-unrefinably atomic.  For completeness, following the diagram, we will show that the implications in the diagram continue to hold even if $\alpha = $ $\tau$-unrefinably atomic.
$$
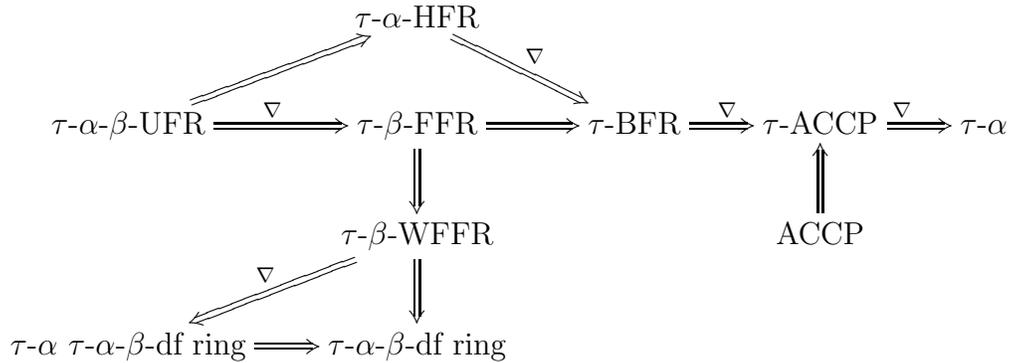
\xymatrix{
            &        \tau\text{-}\alpha \text{-HFR} \ar@{=>}^{\nabla}[dr]     &             &                  &                 \\
\tau\text{-}\alpha\text{-} \beta \text{-UFR} \ar@{=>}[ur] \ar@{=>}^{\nabla}[r]  & \tau\text{-}\beta \text{-FFR} \ar@{=>}[r] \ar@{=>}[d]  & \tau\text{-BFR} \ar@{=>}[r]^{\nabla}& \tau\text{-ACCP} \ar@{=>}^{\nabla}[r]& \tau\text{-}\alpha\\
            & \tau\text{-}\beta\text{-WFFR} \ar@{=>}[d] \ar@{=>}[dl]_{\nabla}                    &              &  \text{ACCP} \ar@{=>}[u]               &                  \\
\tau\text{-}\alpha\  \tau\text{-}\alpha\text{-}\beta \text{-df ring} \ar@{=>}[r]           & \tau\text{-}\alpha\text{-}\beta \text{-df ring} &
            }$$
\begin{theorem} \label{thm: usual diagram} Let $R$ be a commutative ring with $1$ and let $\tau$ be a symmetric relation on $R^{\#}$.  Let $\beta \in \{$ associate, strongly associate, very strongly associate $\}$.  Then we have the following.
\\
(1) If $R$ is a $\tau$-unrefinably atomic-$\beta$-UFR, then $R$ is a $\tau$-unrefinably atomic-HFR.
\\
(2) If $\tau$ is refinable and $R$ is a $\tau$-unrefinably atomic-UFR, then $R$ is a $\tau$-$\beta$-FFR.
\\
(3) If $\tau$ is refinable and $R$ is a $\tau$-unrefinably atomic-HFR, then $R$ is a $\tau$-BFR.
\\
(4) If $R$ is a $\tau$-$\beta$-WFFR, then $R$ is a $\tau$-unrefinably atomic-$\beta$-df ring.
\\
(5) If $\tau$ is refinable and $R$ is a $\tau$-$\beta$-WFFR, then $R$ is a $\tau$-unrefinably atomic $\tau$-unrefinably atomic-$\beta$-df ring.
\\
(6) If $\tau$ is refinable and $R$ satisfies $\tau$-ACCP, then $R$ is $\tau$-unrefinably atomic.
\end{theorem}
\begin{proof} (1) Let $a\in R$ be a non-unit.  Then there is a unique $\tau$-unrefinably atomic factorization $a=\lambda a_1 \cdots a_n$ up to rearrangement and $\beta$.  Any other $\tau$-unrefinably atomic factorization certainly has the same length, $n$. Hence $R$ is a $\tau$-unrefinably atomic-HFR.
\\
\indent (2) Let $\tau$ is refinable and let $R$ is a $\tau$-unrefinably atomic-UFR.  Let $a\in R$ be a non-unit.  Then any $\tau$-factorization of $a$ can be $\tau$-refined into a $\tau$-unrefinably atomic factorization.  By hypothesis, there is only one unique $\tau$-unrefinably atomic factorization of $a$, say $a=\lambda a_1 \cdots a_n$.  Hence all $\tau$-factorizations of $a$ come from some grouping as a product of the elements from the set $\{a_i \}_{i=1}^{n}$, up to $\beta$.  There are only $2^n$ ways to do this, so this serves as a bound on the number of $\tau$-factorizations up to rearrangement and $\beta$, showing $R$ is a $\tau$-$\beta$-FFR.
\\
\indent (3) Let $\tau$ be refinable and let $R$ be a $\tau$-unrefinably atomic-HFR.  Let $a\in R$ be a non-unit.  $R$ is $\tau$-unrefinably atomic, so there is a $\tau$-unrefinably atomic factorization, say $a=\lambda a_1 \cdots a_n$.  Then since any $\tau$-factorization can be $\tau$-refined to a $\tau$-unrefinably atomic factorization, which must have length $n$, we can see that $n$ will serve as an upper bound on the length of any $\tau$-factorization of $a$.  This shows $R$ is a $\tau$-BFR.
\\
\indent (4) Let $R$ be a $\tau$-WFFR and let $a\in R$ be a non-unit.  Then $a$ has a finite number of $\tau$-divisors up to $\beta$.  Then certainly $a$ has a finite number of $\tau$-unrefinably atomic $\tau$-divisors up to $\beta$.  Hence $R$ is a $\tau$-unrefinably atomic-$\beta$-divisor finite ring.
\\
\indent (5) If $\tau$ is refinable and $R$ is a $\tau$-$\beta$-WFFR.  We have already seen that $R$ is a $\tau$-unrefinably atomic-$\beta$-divisor finite ring.  We need to see that $R$ is $\tau$-unrefinably atomic.  In \cite{Mooney}, it proven that, for $\tau$-refinable, a $\tau$-$\beta$-WFFR satisfies $\tau$-ACCP.  By (6) in the present theorem, this will imply that $R$ is $\tau$-unrefinably atomic.
\\
\indent (6) In \cite[Theorem 4.1]{Mooney}, the author proves that $\tau$-ACCP with $\tau$ refinable implies that $R$ is $\tau$-very strongly atomic.  By Theorem \ref{thm: unrefinable}, a $\tau$-very strongly atomic factorization is certainly $\tau$-unrefinably atomic, proving the claim.  Alternatively, we prove a stronger version of this later in Theorem \ref{thm: ACCP}.
\end{proof}

\section{$\tau$-Complete Factorizations Definitions} \label{sec: complete}
Another approach to factorization studied in the domain case is that of $\tau$-complete factorization.  In some ways, this notion is more natural.  The idea behind complete factorization is simply to factor and element as far as possible.  One says a factorization is complete when it is no longer possible to replace one of the factors with a strictly longer factorization.  In the $\tau$-factorization case, we see $\tau$-complete factorizations have several nice consequences.  Many of the properties such as $\tau$ being divisive, multiplicative, refinable, combinable, associate preserving are no longer necessary for many of the major desirable theorems to hold.  
\\
\indent We begin with some definitions.  Recall that a \emph{$\tau$-refinement} of a $\tau$-factorization $\lambda a_1 \cdots a_n$ is a $\tau$-factorization of the form
$$(\lambda \lambda_1 \cdots \lambda_n)b_{11}\cdots b_{1m_1}\cdot b_{21}\cdots b_{2m_2} \cdots b_{n1} \cdots b_{nm_n}$$
where $a_i=\lambda_ib_{i_1}\cdots b_{i_{m_i}}$ is a $\tau$-factorization for each $i$.  A \emph{$\tau$-complete factorization} is a $\tau$-factorization that cannot be $\tau$-refined into a longer $\tau$-factorization.  $R$ is said to be $\tau$-complete if every non-unit has a $\tau$-complete factorization.  $R$ is said to be \emph{$\tau$-completeable} (resp. \emph{$\tau$-atomicable}, \emph{$\tau$-strongly-atomicable}, \emph{$\tau$-m-atomicable}, \emph{$\tau$-unrefinably atomicable}, \emph{$\tau$-very strongly atomicable}) if every $\tau$-factorization can be $\tau$-refined to a $\tau$-complete (resp. $\tau$-atomic, $\tau$-strongly atomic, $\tau$-m-atomic, $\tau$-unrefinably atomic, $\tau$-very strongly atomic) factorization.  Note that sometimes atomizable is used instead of atomicable, we will use the two interchangeably.
\\
\indent Let $\alpha \in \{ $completable, atomicable, strongly atomicable, m-atomicable, unrefinably atomicable, very strongly atomicable$ \}$ and $\beta\in \{$ associate, strong associate, very strong associate $\}$.  If $\alpha =$ completable, set $\alpha'=$ complete.  If $\alpha=$ atomicable (resp. strongly atomicable, m-atomicable, unrefinably atomicable, very strongly atomicable), set $\alpha'=$ atomic (resp. strongly atomic, m-atomic, unrefinably atomic, very strongly atomic).
\\
\indent We then say $R$ is a \emph{$\tau$-$\alpha$-$\beta$-unique factorization ring (UFR)} if (1) $R$ is $\tau$-$\alpha$ and (2) if $a=\lambda\cdot a_1 \cdots a_n=\mu b_1 \cdots$ are two $\tau$-$\alpha'$ factorizations of a non-unit $a \in R$, then $n=m$ and after re-ordering, if necessary, $a_i$ and $b_i$  are $\beta$ for all $i\in \{1, \ldots , n\}$. $R$ is a \emph{$\tau$-$\alpha$-$\beta$-half factorization ring or half factorial ring (HFR)} if (1) $R$ is $\tau$-$\alpha$ and (2) if $a=\lambda\cdot a_1 \cdots a_n=\mu b_1 \cdots$ are two $\tau$-$\alpha'$ factorizations of a non-unit $a \in R$, then $n=m$.  We say that $R$ is a \emph{$\tau$-complete-$\beta$-finite factorization ring (FFR)} (resp. \emph{$\tau$-complete-bounded factorization ring (BFR)}) if for each non-unit $a\in R$, there are only a finite number of $\tau$-complete factorizations of $a$ up to reordering and $\beta$ (resp. there is a natural number $N(a)$ so that for each $\tau$-complete factorization $a=\lambda a_1 \cdots a_n$, $n \leq N(a)$).  We say $R$ is a \emph{$\tau$-complete-$\beta$-divisor finite ring or $\tau$-$\beta$-cdf ring} if for every non-unit $a\in R$ there are a finite number of divisors up to $\beta$, which appear in a $\tau$-complete factorization of $\beta$.
\begin{theorem} \label{thm: complete-atoms} Let $R$ be a commutative ring with $1$, $\tau$ a symmetric relation on $R^{\#}$.  Let $a=\lambda a_1 \cdots a_n$ be a $\tau$-factorization.  We consider the following statements.
\\
(1) This is a $\tau$-very strongly atomic factorization.
\\
(2) This is a $\tau$-unrefinably atomic factorization.
\\
(3) This is a $\tau$-complete factorization.
\\
(4) This is a $\tau$-strongly atomic factorization.
\\
(5) This is a $\tau$-m-atomic factorization.
\\
(6) This is a $\tau$-atomic factorization.
\\
\indent Let $\nabla$ represent $\tau$ being refinable and $\approx$ represent $R$ is strongly associate.  Then we have the following relationship between the different factorizations.
$$\xymatrix{
\tau\text{-very strongly irred.} \ar@{=>}[r] &\tau\text{-unrefinably irred.}\ar@{=>}[d]\ar@{=>}[dr] \ar@{=>}[r]& \tau\text{-strongly irred.} \ar@{=>}[r]& \tau \text{-irred.}\\
& \tau \text{-complete}\ar@{=>}[ur]_>{\nabla} \ar@/_1pc/[u]_{\nabla} \ar@{=>}[r]^{\nabla}& \tau\text{-m-irred.}\ar@{=>}[u]_{\approx}\ar@{=>}[ur]  & &}$$
\end{theorem}
\begin{proof}Many of these implications were shown in \cite[Theorem 3.9]{Mooney} and are immediate from Theorem \ref{thm: unrefinable}.  We need only prove the arrows entering and exiting from the $\tau$-complete factorizations. 
 %We begin with $(1) \Rightarrow (2)$.  This is immediate since a $\tau$-very strongly atomic element is $\tau$-unrefinably atomic.  
\\
\indent $(2) \Rightarrow (3)$  If $a=\lambda a_1 \cdots a_n$ is a $\tau$-unrefinably atomic factorization, then $a_i$ is $\tau$-unrefinably irreducible and hence has only trivial $\tau$-factorizations.  This means there simply are no refinements of $a_i$ which can possibly increase the length of the factorization making the factorization $\tau$-complete.  If $\tau$ is refinable, we show $(3) \Rightarrow (2)$.  If $a=\lambda a_1 \cdots a_n$ is a complete factorization, then if any $a_i$ had a non-trivial $\tau$-factorization $a_i=\mu a_{i1} \cdots a_{in_i}$ with $n_i \geq 2$, then $a=(\lambda \mu) a_1 \cdots a_{i-1} a_{i1} \cdots a_{in_i} a_{i+1} \cdots a_n$ is a $\tau$-factorization with length $n-1+n_i \geq n+1$ since $n_i \geq 2$ contradicting the fact that the factorization was complete. 
\\
%\indent $(2) \Rightarrow (4)$ If $a=\lambda a_1 \cdots a_n$ is a $\tau$-unrefinably atomic factorization, then we show 
\indent If $\tau$ is refinable, then $(3) \Rightarrow (4)$.  Let $a=\lambda a_1 \cdots a_n$ be a $\tau$-complete factorization.  We show that $a_i$ is $\tau$-strongly atomic for all $1\leq i \leq n$.  Suppose there $a_i$ is not $\tau$-strongly atomic.  Then there is a $\tau$-factorization $a_i=\mu b_1 \cdots b_m$ such that $a_i \not \approx b_j$ for any $1\leq j \leq m$.  In particular, $m\geq 2$, or else we have $a_i=\mu b_1$ and $a_i \approx b_1$, a contradiction.  Because $\tau$ is refinable, we can refine the factorization into 
$$a=(\lambda \mu)a_1 \cdots a_{i-1} b_1 \cdots b_m a_{i+1} \cdots a_n.$$
This is a $\tau$-factorization of strictly longer length contradicting the assumption that the factorization was $\tau$-complete.
\\
\indent If $\tau$ is refinable, then $(3) \Rightarrow (5)$.  Let $a=\lambda a_1 \cdots a_n$ be a $\tau$-complete factorization.  We show that $a_i$ is $\tau$-m-atomic for all $1\leq i \leq n$.  Suppose there $a_i$ is not $\tau$-m-atomic.  Then there is a principal ideal generated by some $b_1 \in R$ such that $b\mid_{\tau} a_i$ and $(a_i) \subsetneq (b_1)$.  Because $b_1\mid_{\tau} a_i$, there exists a $\tau$-factorization of the form $a_i=\mu b_1 \cdots b_m$.  In particular, $m\geq 2$, or else we have $a_i=\mu b_1$ and $a_i \sim b_1$, a contradiction.  Because $\tau$ is refinable, we can refine the factorization into 
$$a=(\lambda \mu)a_1 \cdots a_{i-1} b_1 \cdots b_m a_{i+1} \cdots a_n.$$
This is a $\tau$-factorization of strictly longer length contradicting the assumption that the factorization was $\tau$-complete.
\end{proof}
We now provide examples to show $\tau$-complete factorizations are indeed distinct. 
\begin{example}\ 
\\
\begin{enumerate}
\item Let $R=\Z/2\Z \times \Z/2\Z$ with $\tau=\{((1,0),(1,0))\}$.
\\\\
\indent Consider the $\tau$-factorization $(1,0)=(1,0)(1,0)$.  This is a $\tau$-m-atomic and $\tau$-strongly atomic factorization, but neither a $\tau$-complete nor $\tau$-unrefinably atomic factorization.  To see this, $(1,0)R$ is a maximal ideal since $R/(1,0)R \cong \Z/2\Z$, a field.  If an ideal is maximal, it is certainly maximal among principal ideals, so it is m-atomic and therefor $\tau$-m-atomic.  $R$ is strongly associate, so we know that the factorization is also $\tau$-strongly atomic.  On the other hand, 
$$(1,0)=(1,0)(1,0)=(1,0)\left((1,0)\cdot (1,0)\right)=(1,0)(1,0)(1,0)$$
gives us a $\tau$-refinement of the factorization which is properly longer, showing it is not a $\tau$-complete factorization.  This also shows that the factorization is not $\tau$-unrefinably atomic.
\item Let $R=\Z/2\Z \times \Z/2\Z$ with $\tau=\{((1,0),(0,1)), ((0,1),(1,0))\}$.
\\\\
\indent Consider the $\tau$-factorization $(0,0)=(1,0)(0,1)$.  There are no non-trivial $\tau$-factorizations of $(1,0)$ or $(0,1)$, this makes the factorization both $\tau$-complete and $\tau$-unrefinably atomic since it cannot be refined into any longer $\tau$-factorization.  On the other hand $(1,0)$ is not $\tau$-very strongly atomic because it fails the $(1,0) \cong (1,0)$ part of the definition.  The factorization $(1,0)=(1,0)(1,0)$ and noting that $(1,0)$ is not a unit in $R$ shows this.  Hence we have a $\tau$-complete and $\tau$-unrefinably atomic factorization which is not $\tau$-very strongly atomic.
\\\\
\item Let $R = \Z$ and let $\tau=\{ (\pm 2,\pm 2),(\pm 3,\pm 3), (\pm 4,\pm 9), (\pm 9,\pm 4) \}$.
\\\\
\indent Then $\tau$ is associate preserving and symmetric, but not refinable.  We consider the factorization $36=4\cdot 9$ and notice that $4$ and $9$ are not even $\tau$-irreducible since $4=2\cdot 2$ and $9=3 \cdot 3$ are $\tau$-factorizations which show $4$ and $9$ are not $\tau$-atomic.  On the other hand, we see that this factorization is $\tau$-complete since $\pm 2 \not \tau \pm 3, \pm 9$ and $\pm 3 \not \tau \pm 2, \pm 4$, so no $\tau$-refinements of this factorization are valid $\tau$-factorizations, so the factorization is $\tau$-complete despite not being $\tau$-atomic, strongly atomic, m-atomic, unrefinably atomic, or very strongly atomic.
\\\\
\indent Examples given in \cite{Valdezleon} show that the other arrows are not reversible even when $\tau = R^{\#}\times R^{\#}$.
\end{enumerate}
\end{example}
As in a series of papers by A. Bouvier, \cite{bouvier71, bouvier72a, bouvier72b, bouvier74}, a commutative ring is said to be \emph{pr\'esimplifiable} if $x=xy$ for some $x,y\in R$ implies $x=0$ or $y\in U(R)$.  A nice property of the various $\tau$-irreducibles defined in \cite{Mooney} is that they all coincide when a ring is pr\'esimplifiable.   When $R$ is pr\'esimplifiable and $\tau$ is refinable, we can add $\tau$-complete factorizations to the list of equivalent $\tau$-irreducible factorizations.  We summarize this in the following theorem which follows relatively easily from Theorem \ref{thm: complete-atoms}.
\begin{theorem}\label{prp: presimplifiable} Let $R$ be pr\'esimplifiable and $\tau$ be a symmetric, refinable relation on $R^{\#}$.  For a $\tau$-factorization $a=\lambda a_1 \cdots a_n$, the following are equivalent.
\\
(1) This is a $\tau$-very strongly atomic factorization.
\\
(2) This is a $\tau$-unrefinably atomic factorization.
\\
(3) This is a $\tau$-complete factorization.
\\
(4) This is a $\tau$-strongly atomic factorization.
\\
(5) This is a $\tau$-m-atomic factorization.
\\
(6) This is a $\tau$-atomic factorization.
\end{theorem}
\begin{proof}In a pr\'esimplifiable ring, all the associate relations coincide.  We have $x\sim y \Rightarrow x \cong y$, in particular $a_i \sim a_i \Rightarrow  a_i \cong a_i$, so we have $(5) \Rightarrow (1)$.  This coupled with what was shown in Theorem \ref{thm: complete-atoms} completes the proof.
\end{proof}
\begin{theorem} \label{lem: complete refinement} Let $R$ be a commutative ring with $1$ and let $\tau$ be a refinable, symmetric relation on $R^{\#}$.  If $a=\lambda a_1 \cdots a_n$ is a $\tau$-factorization and $a_i=\lambda_i b_{i1} \cdots b_{im_i}$ are $\tau$-complete factorizations for $1 \leq i \leq n$, then 
\begin{equation} \label{eq: refinement}
a=(\lambda \lambda_1 \cdots \lambda_n) b_{11} \cdots b_{1m_1} b_{21} \cdots b_{2m_2} \cdots b_{n1}\cdots b_{nm_n}
\end{equation}
is a $\tau$-complete factorization.
\end{theorem}
\begin{proof} Because $\tau$ is refinable, the factorization given in \ref{eq: refinement} is certainly a $\tau$-factorization.  It remains to be seen that this factorization is $\tau$-complete.  Taking the notation from the statement of the theorem, we suppose there is a $\tau$-refinement of $b_{ij}$ for some $1\leq i \leq n$ and $1\leq j \leq m_i$ of the form $b_{ij}=\mu c_1 \cdots c_k$ which makes the factorization in equation \eqref{eq: refinement} properly longer.  This also yields a $\tau$-refinement of the $\tau$-complete factorization $a_i=\lambda_i b_{i1} \cdots b_{im_i}$
$$a_i= (\lambda_i\mu) b_{i1} \cdots b_{i(j-1) } c_1 \cdots c_k b_{i(j+1)} \cdots b_{im_i}$$
into a $\tau$-factorization which is properly longer, a contradiction, completing the proof.
\end{proof}

%\begin{theorem}\label{thm: vs atomic}
%A $\tau$-very strongly atomic factorization is $\tau$-complete.  If $\tau$ is refinable and $R$ is pr\'esimplifiable, then a $\tau$-complete factorization is $\tau$-very strongly atomic.
%\end{theorem}
%\begin{proof}Let $a=\lambda a_1 \cdots a_n$ be a $\tau$-very strongly atomic factorization.  Then for all $a_i$, there are only trivial $\tau$-factorizations, and hence any refinement maintains the same length, showing the factorization to be $\tau$-complete.  Converesely, let $a=\lambda a_1 \cdots a_n$ be a $\tau$-complete factorization.  Because $R$ is assumed to be pr\'esimplifiable, we certainly have $a_i\cong a_i$ for all $i$.  So we suppose there were some $1 \leq i_0\leq n$ such that there is a non-trivial $\tau$-factorization $a_{i_0}=\mu b_1 \cdots b_m$ with $m \geq 2$.  We assumed $\tau$ is refinable, so $ a=(\lambda \mu)a_1 \cdots a_{i_0-1} b_1 \cdots b_m a_{i_0+1} \cdots a_n$ is a $\tau$-factorization of strictly longer length, a contradiction.  Thus each $a_i$ is very strongly atomic as desired.
%\end{proof}
\section{$\tau$-Complete Factorization Relationships}
In this section, we look at the relationship between the $\tau$-complete (completable) factorizations defined in Section \ref{sec: complete} and the $\tau$-atomic (atomicable) (resp. strongly atomic (atomicable), m-atomic (atomicable), unrefinably atomic (atomicable), very strongly atomic (atomicable)) factorizations defined in \cite{Mooney}.  
\\
\indent Let $\alpha \in \{ $completable, atomicable, strongly atomicable, m-atomicable, unrefinably atomicable, very strongly atomicable$ \}$.  If $\alpha =$ completable, set $\alpha'=$ complete.  If $\alpha=$ atomicable (resp. strongly atomicable, m-atomicable, unrefinably atomicable, very strongly atomicable), set $\alpha'=$ atomic (resp. strongly atomic, m-atomic, unrefinably atomic, very strongly atomic).
\begin{theorem} \label{thm:big theorem} Let $R$ be a commutative ring with 1.  Let $\tau$ be a symmetric relation on $R^{\#}$.\\
(1) If $R$ is $\tau$-very strongly atomic, then $R$ is $\tau$-unrefinably atomic.\\
(2) If $R$ is $\tau$-unrefinably atomic, then $R$ is $\tau$-complete.\\
(3) If $\tau$ is refinable and $R$ is $\tau$-complete, then $R$ is $\tau$-unrefinably atomic.\\
(4) If $\tau$ is refinable and $R$ is $\tau$-complete, then $R$ is $\tau$-strongly atomic.\\
(5) If $\tau$ is refinable and $R$ is $\tau$-complete, then $R$ is $\tau$-m-atomic.\\
(6) If $R$ is $\tau$-m-atomic and strongly associate, then $R$ is $\tau$-strongly atomic.\\
(7) If $R$ is $\tau$-m-atomic, then $R$ is $\tau$-atomic.\\
(8) If $R$ is $\tau$-strongly atomic, then $R$ is $\tau$-atomic.\\
(9) If $R$ is $\tau$-very strongly atomicable, then $R$ is $\tau$-unrefinably atomicable.\\
(10 If $R$ is $\tau$-unrefinably atomicable, then $R$ is $\tau$-completable.\\
(11) If $\tau$ is refinably and $R$ is $\tau$-completable, then $R$ is $\tau$-unrefinably atomicable.\\
(12) If $\tau$ is refinable and $R$ is $\tau$-completable, then $R$ is $\tau$-strongly atomicable.\\
(13) If $\tau$ is refinable and $R$ is $\tau$-completeable, then $R$ is $\tau$-m-atomicable.\\
(14) If $R$ is $\tau$-m-atomicable and strongly associate, then $R$ is $\tau$-strongly atomicable.\\
(15) If $R$ is $\tau$-m-atomicable, then $R$ is $\tau$-atomicable.\\
(16) If $R$ is $\tau$-strongly atomicable, then $R$ is $\tau$-atomicable.\\
%(2) If $\tau$ is refinable and $R$ is $\tau$-complete, then $R$ is $\tau$-very strongly atomic.\\
(17) If $R$ is $\tau$-$\alpha$, then $R$ is $\tau$-$\alpha'$.\\
(18) If $\tau$ is refinable and $R$ is $\tau$-$\alpha'$, then $R$ is $\tau$-$\alpha$.\\
\indent If $R$ is pr\'esimplifiable, then (1)-(8) are equivalent and (9)-(16) are equivalent.
%(5) If $R$ is $\tau$-very strongly atomicable, then $R$ is $\tau$-completable.\\
\end{theorem}
\begin{proof} (1) (resp. (2), (3), (4), (5), (6), (7), (8)) Let $a \in R^{\#}$.  Let $a\in R$ be a non-unit.  Because $R$ is $\tau$-very strongly atomic (resp. $\tau$-unrefinably atomic, $\tau$-complete, $\tau$-complete, $\tau$-complete, $\tau$-m-atomic, $\tau$-m-atomic,  $\tau$-strongly atomic) we have a $\tau$-factorization $a=\lambda a_1 \cdots a_n$ which is $\tau$-very strongly atomic (resp. $\tau$-unrefinably atomic, $\tau$-complete, $\tau$-complete, $\tau$-complete, $\tau$-m-atomic, $\tau$-m-atomic,  $\tau$-strongly atomic).  We now apply Theorem \ref{thm: complete-atoms} and the hypothesis to conclude that this factorization is $\tau$-unrefinably atomic (resp. $\tau$-complete, $\tau$-unrefinably atomic, $\tau$-strongly atomic, $\tau$-m-atomic, $\tau$-strongly atomic, $\tau$-atomic, $\tau$-atomic), proving the claim.
\\
\indent (9) (resp. (10), (11), (12), (13), (14), (15), (16)) Let $a=\lambda a_1 \cdots a_n$ be a $\tau$-factorization.  Because $R$ is $\tau$-very strongly atomicable (resp. $\tau$-unrefinably atomicable, $\tau$-completable, $\tau$-completeable, $\tau$-completable, $\tau$-m-atomicable, $\tau$-m-atomicable,  $\tau$-strongly atomicable) we can $\tau$-refine this factorization into a $\tau$-factorization $a=\lambda b_1 \cdots b_m$ which is $\tau$-very strongly atomic (resp. $\tau$-unrefinably atomic, $\tau$-complete, $\tau$-complete, $\tau$-complete, $\tau$-m-atomic, $\tau$-m-atomic,  $\tau$-strongly atomic).  By Theorem \ref{thm: complete-atoms} and the hypothesis, this factorization is $\tau$-unrefinably atomic (resp. $\tau$-complete, $\tau$-unrefinably atomic, $\tau$-strongly atomic, $\tau$-m-atomic, $\tau$-strongly atomic, $\tau$-atomic, $\tau$-atomic).  This proves any $\tau$-factorization can be $\tau$-refined to a $\tau$-unrefinably atomic (resp. $\tau$-complete, $\tau$-unrefinably atomic, $\tau$-strongly atomic, $\tau$-m-atomic, $\tau$-strongly atomic, $\tau$-atomic, $\tau$-atomic) factorization as desired.
\\
\indent (17) Let $R$ be $\tau$-$\alpha$.  Let $a\in R^{\#}$.  Then $a=1 \cdot a$ is a $\tau$-factorization and thus it can be $\tau$-refined into a $\tau$-$\alpha'$-factorization.  Thus for any non-zero, non-unit, we can find a $\tau$-$\alpha'$-factorization proving $R$ is $\alpha'$ as desired.
\\
\indent (18) Let $R$ be $\alpha'$ with $\tau$-refinable.  Let $a=\lambda a_1 \cdots a_n$ be a $\tau$-factorization.  $R$ is $\alpha'$, so let 
$$a_i = \lambda_i b_{1i} \cdots b_{m_i i}$$
be a $\tau$-$\alpha'$ factorization of $a_i$ for $1 \leq i \leq n$.  By hypothesis $\tau$-is refinable, so 
\begin{equation} \label{eq:factor} a=(\lambda \lambda_1 \cdots \lambda_n) b_{11} \cdots b_{m_1 1} \cdots b_{12} \cdots b_{m_2 2} \cdots b_{1n} \cdots b_{m_nn}
\end{equation}
is a $\tau$-factorization.  Furthermore, for $\alpha' \in \{$ atomic, strongly atomic, m-atomic, unrefinably atomic very strongly atomic $\}$, we can immediately conclude this factorization is $\tau$-$\alpha'$, proving the claim.  
\\
\indent For $\alpha'=$ complete, we apply Theorem \ref{lem: complete refinement} to Equation \eqref{eq:factor} to see that for a refinable $\tau$, a $\tau$-factorization comprised of $\tau$-complete parts remains $\tau$-complete.  
\\
\indent The final sentence follows from an application of Theorem \ref{prp: presimplifiable} since all these factorization types coincide.
\end{proof}
%\indent Let $R$ be a commutative ring with $1$ and $\tau$ a symmetric relation on $R^{\#}$.  
The following diagrams help summarize the relationship between the above properties.  Where $\nabla$ indicates $\tau$ is refinable, and $\approx$ indicates $R$ is a strongly associate ring.  $R$ is 
$$\xymatrix{
\tau\text{-very strongly irred.} \ar@{=>}[r] &\tau\text{-unrefinably irred.}\ar@{=>}[d]\ar@{=>}[dr] \ar@{=>}[r]& \tau\text{-strongly irred.} \ar@{=>}[r]& \tau \text{-irred.}\\
& \tau \text{-complete}\ar@{=>}[ur]_>{\nabla} \ar@/_1pc/[u]_{\nabla} \ar@{=>}[r]^{\nabla}&\tau\text{-m-irred.}\ar@{=>}[u]_{\approx}\ar@{=>}[ur]  & &}$$

%$$
%\xymatrix{
%R \text{ is }\tau\text{-very strongly atomic}\ar@{=>}[d]\ar@{=>}[dr] \ar@{=>}[r]&R \text{ is } \tau\text{-strongly atomic} \ar@{=>}[r]&R \text{ is } \tau \text{-atomic}\\
%R \text{ is }\tau \text{-complete}\ar@{=>}[ur]^<{\star} \ar@{=>}[r]^{\star}& R \text{ is }\tau\text{-m-atomic}\ar@{=>}[u]_{\dagger}\ar@{=>}[ur]  & &}
%$$

$$\xymatrix{
\tau\text{-v.s. atomicable} \ar@{=>}[r] &\tau\text{-unref. atomicable}\ar@{=>}[d]\ar@{=>}[dr] \ar@{=>}[r] & \tau\text{-s. atomicable} \ar@{=>}[r]& \tau \text{-atomicable}\\
& \tau \text{-completable}\ar@{=>}[ur]_>{\nabla} \ar@/_1pc/[u]_{\nabla} \ar@{=>}[r]^{\nabla}& \tau\text{-m-atomicable}\ar@{=>}[u]_{\approx}\ar@{=>}[ur]  & &}$$

\indent Let $\alpha \in \{$ completable, atomicable, strongly atomicable, m-atomicable, unrefinably atomicable, very strongly atomicable $\}$.  If $\alpha =$ completable, set $\alpha'=$ complete.  If $\alpha=$ atomicable (resp. strongly atomicable, m-atomicable, unrefinably atomicable, very strongly atomicable), set $\alpha'=$ atomic (resp. strongly atomic, m-atomic, unrefinably atomicable, very strongly atomic). Let $\nabla$ indicate $\tau$-refinable.
$$
\xymatrix{
R \text{ is }\tau\text{-}\alpha \ar@{=>}[d]  \\ %& & R \text{ is }\tau\text{-}\alpha'\ar@{=>}^{\star}[d]\\
R \text{ is }\tau\text{-} \alpha' \ar@/_1pc/_{\nabla}[u] } %& & R \text{ is }\tau\text{-}\alpha }
$$
\indent The following corollary is an immediate consequences of the definitions and Theorem \ref{thm:big theorem} parts (17) and (18).  The proof is clear and thus has been omitted. 
\begin{corollary}\label{cor: able to not} Let $R$ be a commutative ring with $1$ and $\tau$ be a symmetric relation on $R^{\#}$.  Let $\beta \in \{$ associate, strongly associate, very strongly associate $\}$. We have the following.
\\
(1) $R$ is a $\tau$-very strongly atomicable-$\beta$-UFR (resp. $\tau$-very strongly atomicable-HFR, $\tau$-very strongly atomicable $\tau$-very strongly atomic-$\beta$ df ring) implies $R$ is a $\tau$-very strongly atomic-$\beta$-UFR (resp. $\tau$-very strongly atomic-HFR, $\tau$-very strongly atomic $\tau$-very strongly atomic-$\beta$ df ring).
If $\tau$ is refinable, then the converses also hold.
\\
(2) $R$ is a $\tau$-unrefinably atomicable-$\beta$-UFR (resp. $\tau$-unrefinably atomicable-HFR, $\tau$-unrefinably atomicable $\tau$-unrefinably atomic-$\beta$ df ring) implies $R$ is a $\tau$-unrefinably atomic-$\beta$-UFR (resp. $\tau$-unrefinably atomic-HFR, $\tau$-unrefinably atomic $\tau$-unrefinably atomic-$\beta$ df ring).
\\
(3) $R$ is a $\tau$-completeable-$\beta$-UFR (resp. $\tau$-completeable-HFR, $\tau$-completeable $\tau$-$\beta$-cdf ring) implies $R$ is a $\tau$-complete-$\beta$-UFR (resp. $\tau$-complete-HFR, $\tau$-complete $\tau$-$\beta$-cdf-ring).
\\
(4) $R$ is a $\tau$-m-atomicable-$\beta$-UFR (resp. $\tau$-m-atomicable-HFR, $\tau$-m-atomicable $\tau$-m-atomic-$\beta$ df ring) implies $R$ is a $\tau$-m-atomic-$\beta$-UFR (resp. $\tau$-m-atomic-HFR, $\tau$-m-atomic $\tau$-m-atomic-$\beta$ df ring).
\\
(5)  $R$ is a $\tau$-strongly atomicable-$\beta$-UFR (resp. $\tau$-strongly atomicable-HFR, $\tau$-strongly atomicable $\tau$-strongly atomic-$\beta$ df ring) implies $R$ is a $\tau$-strongly atomic-$\beta$-UFR (resp. $\tau$-strongly atomic-HFR, $\tau$-atomic $\tau$-strongly atomic-$\beta$ df ring).
\\
(6) $R$ is a $\tau$-atomicable-$\beta$-UFR (resp. $\tau$-atomicable-HFR, $\tau$-atomicable $\tau$-atomic-$\beta$-df ring) implies $R$ is a $\tau$-atomic-$\beta$-UFR (resp. $\tau$-atomic-HFR, $\tau$-atomic $\tau$-atomic-$\beta$ df ring).
\end{corollary}

\begin{theorem} \label{thm: ACCP} Let $R$ be a commutative ring with $1$ and $\tau$ be a refinable, associate preserving, symmetric relation on $R^{\#}$.  If $R$ satisfies $\tau$-ACCP, then
\\
(1) $R$ is $\tau$-very strongly atomic and $\tau$-very strongly atomicable.
\\
(2) $R$ is $\tau$-unrefinably atomic and $\tau$-unrefinably atomicable
\\
(3) $R$ is $\tau$-complete and $\tau$-completable.
\\
(4) $R$ is $\tau$-m-atomic and $\tau$-m-atomicable.
\\
(5) $R$ is $\tau$-strongly atomic and $\tau$-strongly atomicable.
\\
(6) $R$ is $\tau$-atomic and $\tau$-atomicable.
\end{theorem}
\begin{proof} It was shown in \cite[Theorem 4.1]{Mooney}, that for $\tau$ refinable and associate preserving if $R$ satisfies $\tau$-ACCP, then $R$ is $\tau$-very strongly atomic.  Hence for each non-unit $a\in R$, there is a $\tau$-very strongly atomic factorization of $a$.  By Theorem \ref{thm: complete-atoms}, this factorization is also $\tau$-unrefinably atomic, $\tau$-complete, $\tau$-m-atomic, $\tau$-strongly atomic, and $\tau$-atomic.  This shows $R$ is a  $\tau$-very strongly atomic, $\tau$-unrefinably atomic, $\tau$-complete, $\tau$-m-atomic, $\tau$-strongly atomic, and $\tau$-atomic ring.  Lastly, using part (18) of Theorem \ref{thm:big theorem} and the fact that $\tau$ is refinable, we see that $R$ is $\tau$-very strongly atomicable, $\tau$-unrefinably atomicable, $\tau$-completable, $\tau$-m-atomicable, $\tau$-strongly atomicable, and $\tau$-atomicable.
\end{proof}
\indent The following theorem shows that we get a very similar finite factorization diagram compared to the one preceding Theorem \ref{thm: usual diagram} using the $\tau$-complete and $\tau$-completable factorizations.  In many cases, using complete factorizations, we can eliminate the usual requirement that $\tau$ be refinable.

\begin{theorem} \label{thm: main diagram} Let $R$ be a commutative ring with $1$ and $\tau$ be a symmetric relation on $R^{\#}$.  Let $\beta \in \{$ associate, strongly associate, very strongly associate $\}$.  Then we have the following.
\\
(1) If $R$ is a $\tau$-complete (resp. completable)-$\beta$-UFR, then $R$ is a $\tau$-complete (resp. completable)-$\beta$-HFR.
\\
(2) If $R$ is a $\tau$-complete (resp. completable)-$\beta$-UFR, then $R$ is a $\tau$-complete-$\beta$-FFR.
\\
(3) If $R$ is a $\tau$-complete (resp. completable)-$\beta$-UFR, then $R$ is a $\tau$-complete (resp. completable) $\tau$-$\beta$-cdf ring.
\\
(4) If $R$ is a $\tau$-complete (resp. completable)-HFR, then $R$ is a $\tau$-complete-BFR.
\\
(5) If $R$ is a $\tau$-complete-$\beta$-FFR, then $R$ is a $\tau$-complete-BFR.
\\
(6) If $R$ is a $\tau$-complete-$\beta$-FFR, then $R$ is a $\tau$-$\beta$-cdf ring.
\\
(7) For $R$ $\tau$-complete and $\tau$ is refinable (resp. For $R$ $\tau$-completable), if $R$ is a $\tau$-complete-$\beta$-FFR, then $R$ is a $\tau$-complete (resp. completable) $\tau$-complete-$\beta$-divisor finite ring.
\\
(8) For $R$ $\tau$-complete and $\tau$ refinable (resp. For $R$ $\tau$-completable), if $R$ is $\tau$-complete-BFR, then $R$ satisfies $\tau$-ACCP.
\\
\indent This yields the following diagram where $\alpha\in \{$ complete, completable $\}$ and $\dagger$ indicates $\tau$ is refinable and $R$ is $\tau$-complete.
$$
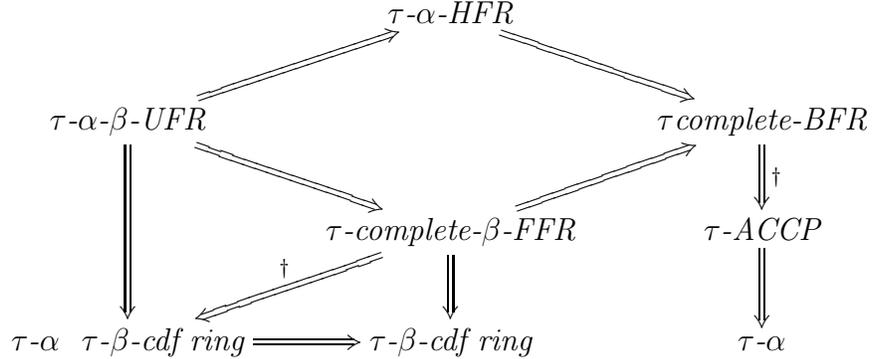
\xymatrix{
            &        \tau\text{-}\alpha \text{-HFR} \ar@{=>}[dr]     &                           \\
\tau\text{-}\alpha \text{-} \beta \text{-UFR} \ar@{=>}[ur] \ar@{=>}[dr]\ar@{=>}[dd]    &   & \tau\text{complete-BFR} \ar@{=>}[d]^{\dagger}\\
            &        \tau\text{-complete-}\beta \text{-FFR} \ar@{=>}[ur]\ar@{=>}[dl]_{\dagger} \ar@{=>}[d]             &      \tau\text{-ACCP} \ar@{=>}[d]                     \\
\tau\text{-}\alpha\ \ \tau \text{-} \beta \text{-cdf ring}		\ar@{=>}[r]  				&   \tau\text{-}\beta\text{-cdf ring}              &          \tau\text{-}\alpha       \\
            }$$
\end{theorem}
\begin{proof} $[(1), (2) \text{ and } (3)]$ Let $R$ be a $\tau$-complete (resp. completable)-$\beta$-UFR.  Then $R$ is $\tau$-complete (resp. completable) by definition.  Furthermore, for a non-unit $a\in R$, if there is precisely one complete $\tau$-complete-factorization up to rearrangement and $\beta$, say $a=\lambda a_1 \cdots a_n$.  Thus certainly the length is unique, proving $R$ is a $\tau$-complete (resp. completable)-$\beta$-HFR.  
\\
\indent This also shows $R$ is a $\tau$-complete-$\beta$-FFR since there is only one $\tau$-complete factorization up to rearrangement and $\beta$.  Furthermore, the only $\tau$-divisors appearing as $\tau$-factors in a complete factorization up to $\beta$ are among the set $\{a_1, \ldots, a_n\}$ and hence there are only finitely many, proving $R$ is a $\tau$-complete (resp. completable) $\tau$-$\beta$-cdf ring.
\\
\indent (4) Let $R$ be a $\tau$-complete (resp. completable)-HFR.  Let $a\in R$ be a non-unit.  Then $a$ has a $\tau$-complete factorization, say $a=\lambda a_1 \cdots a_n$.  We then set $N(a)=n$.  Given any $\tau$-complete factorization, we know it has length $n$, so $R$ is a $\tau$-complete-BFR.
\\
\indent (5) Let $R$ be a $\tau$-complete-$\beta$-FFR.  Let $a\in R$ be a non-unit.  Then there are a finite number of $\tau$-complete factorizations of $a$ up to rearrangement and $\beta$.  Set $N(a)$ equal to the length of the largest such $\tau$-complete factorization.  Given any $\tau$-complete factorization of $a$, it is either among the given factorizations, or there is a rearrangement and switching of $\beta$.  In any case, the $\tau$-factorization has length less than $N(a)$, proving $R$ is a $\tau$-complete-BFR.
\\
\indent (6) Let $R$ be a $\tau$-complete-$\beta$-FFR, and let $a\in R$ be a non-unit.  There are a finite number of $\tau$-complete factorizations up to rearrangement and $\beta$.  Each of these $\tau$-complete factorizations has a finite length.  Thus the set of all $\tau$-factors which occur as a divisor in some $\tau$-complete factorization of $a$ must be finite.
\\
\indent (7) This is immediate from (6) and the definitions.
\\
\indent (8) Let $R$ be a $\tau$-complete-BFR.  We suppose there is an ascending chain $(a_1) \subsetneq (a_2) \subsetneq \cdots \subsetneq (a_i) \subsetneq \cdots$ of principal ideals such that $a_{i+1} \mid_{\tau} a_i$.  Let $N(a_1)$ be the bound on the length of the $\tau$-complete factorizations of $a_1$.  We have $\tau$-factorizations
$a_{i}=\lambda_i a_{i+1} b_{i1}\cdots b_{in_i}$ for each $i$.  We note here that $n_i \geq 1$ or else we would have $(a_i)=(a_{i+1})$.  Because $\tau$ is refinable, we can create $\tau$-factorizations as follows:
$$a_1=\lambda_1 a_{2} b_{11}\cdots b_{1n_1}=\lambda_1\lambda_2 a_{3} b_{21}\cdots b_{2n_2}b_{11}\cdots b_{1n_1}= \cdots .$$
After $N(a_1)$ iterations, we will arrive at a $\tau$-factorization, $\dagger$, of length at least $N(a_1)$ since at each stage the length increases by at least $1$.  Now $\tau$ is refinable and $R$ is $\tau$-complete, so we apply Theorem \ref{lem: complete refinement} to $\tau$-refine the $\tau$-factorization, $\dagger$, of length $N(a_1)$ into a $\tau$-complete factorization. (resp. Because $R$ is $\tau$-completable, we can $\tau$-refine the factorization, $\dagger$, into a $\tau$-complete factorization.) This can only increase the length of the factorization which contradicts the fact that $N(a_1)$ is the bound on the length of $\tau$-complete factorizations of $a_1$.  
%\\
%\indent In fact, we see from the proof that the bound on the length of $\tau$-complete-factorizations of a non-unit $a\in R$ serves as a bound on the length of any properly ascending chain of principal ideals ascending from $a$ such that each principal ideal generator $\mid_\tau$ the preceeding one.
\end{proof}
We get a similar analogue for $\tau$-atomicable, strongly atomicable, m-atomicable, unrefinably atomicable, very strongly atomicable rings.
\begin{theorem} \label{thm: ufr to hfr} Let $R$ be a commutative ring with $1$ and $\tau$ be a symmetric relation on $R^{\#}$.  Let $\alpha \in \{ $atomicable, strongly atomicable, m-atomicable, unrefinably atomicable, very strongly atomicable $\}$.  If $\alpha=$ atomicable (resp. strongly atomicable, m-atomicable, unrefinably atomicable, very strongly atomicable), set $\alpha'=$ atomic (resp. strongly atomic, m-atomic, unrefinably atomic, very strongly atomic).  Let $\beta \in \{ $associate, strongly associate, very strongly associate $\}$.  Then we have the following.
\\
(1) If $R$ is a $\tau$-$\alpha$-$\beta$-UFR, then $R$ is a $\tau$-$\alpha$-HFR.
\\
(2) If $\tau$ is refinable and $R$ is a $\tau$-$\alpha$-$\beta$-UFR, then $R$ is a $\tau$-$\beta$-FFR.
\\
(3) If $R$ is a $\tau$-$\alpha$-$\beta$-UFR, then $R$ is a $\tau$-$\alpha$ $\tau$-$\alpha'$-$\beta$-divisor finite ring.
\\
(4) If $\tau$ is refinable and $R$ is a $\tau$-$\alpha$-$\beta$-HFR, then $R$ is a $\tau$-BFR.
\\
(5) If $R$ is a $\tau$-$\alpha$ $\tau$-$\alpha'$-$\beta$-divisor finite ring, then $R$ is a $\tau$-$\alpha'$-$\beta$-divisor finite ring.
\\
(6) If $\tau$ is refinable and $R$ is a $\tau$-$\beta$-WFFR, then $R$ is a $\tau$-$\alpha$ $\tau$-$\alpha'$-$\beta$-divisor finite ring.
\\
The following diagram summarizes the relationship between these new finite factorization properties that $R$ might possess, where $\nabla$ indicates $\tau$ is refinable:
$$
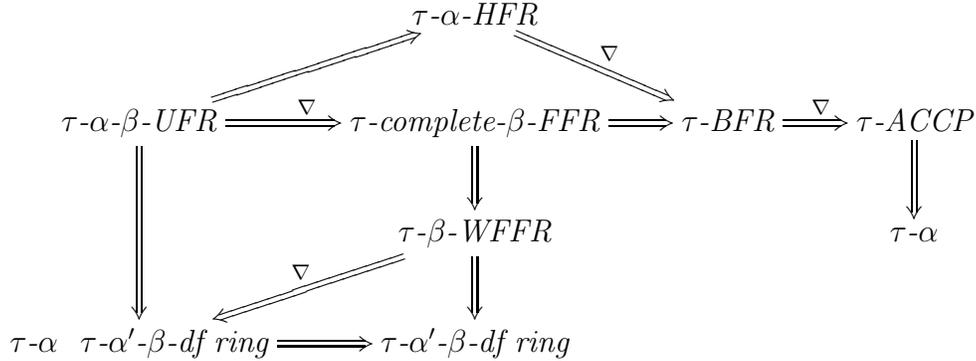
\xymatrix{
            &        \tau\text{-}\alpha \text{-HFR} \ar@{=>}[dr]^{\nabla}     &                   &        \\
\tau\text{-}\alpha \text{-} \beta \text{-UFR} \ar@{=>}[ur] \ar@{=>}[r]^{\nabla} \ar@{=>}[dd]    &  \tau\text{-complete-}\beta \text{-FFR} \ar@{=>}[r] \ar@{=>}[d]  & \tau\text{-BFR} \ar@{=>}[r]^{\nabla} & \tau\text{-ACCP} \ar@{=>}[d] \\
            &         \tau\text{-}\beta\text{-WFFR}  \ar@{=>}[dl]_{\nabla}     \ar@{=>}[d]       &           &        \tau\text{-}\alpha        \\
\tau\text{-}\alpha\ \ \tau \text{-}\alpha'\text{-} \beta \text{-df ring}		\ar@{=>}[r]  				&   \tau\text{-}\alpha'\text{-}\beta\text{-df ring}              &             &   \\
            }$$
\end{theorem}
\begin{proof} (1) Suppose $R$ is a $\tau$-$\alpha$-$\beta$-UFR.  By hypothesis $R$ is $\tau$-$\alpha$.  Let $a$ be a non-unit.  Suppose $a=\lambda a_1 \cdots a_n= \mu b_1 \cdots b_m$ were two $\tau$-$\alpha'$ factorizations with different lengths.  Then this contradicts the fact that $R$ is a $\tau$-$\alpha$-$\beta$-UFR, and proves the theorem.
\\
\indent (2) Suppose $\tau$ is refinable and $R$ is a $\tau$-$\alpha$-$\beta$-UFR.  Then $R$ is a $\tau$-$\alpha'$-$\beta$-UFR by Corollary \ref{cor: able to not}.  We then apply Theorem \ref{thm: usual diagram} and the hypothesis that $\tau$ is refinable to conclude that $R$ is also a $\tau$-$\beta$-FFR.
\\
\indent (3) Let $R$ be a $\tau$-$\alpha$-$\beta$-UFR.  Then $R$ is $\tau$-$\alpha$ and again, $R$ is a $\tau$-$\alpha'$-$\beta$-UFR by Corollary \ref{cor: able to not}.  Let $a\in R$ be a non-unit.  Let $a=\lambda a_1 \cdots a_n$ be the unique $\tau$-$\alpha'$-factorization up to $\beta$.  Then $a_1, \ldots, a_n$ are the only $\tau$-$\alpha'$ divisors of $a$ up to $\beta$, which proves $R$ is a $\tau$-$\alpha$ $\tau$-$\alpha'$-$\beta$-divisor finite ring.
\\
\indent (4) Let $\tau$ be refinable and $R$ be a $\tau$-$\alpha$-$\beta$-HFR.  Then by Corollary \ref{cor: able to not}, $R$ is a $\tau$-$\alpha'$-$\beta$-HFR.  We then apply the hypothesis that $\tau$ is refinable and Theorem \ref{thm: usual diagram} to see that $R$ is a $\tau$-BFR.
\\
\indent (5) This is immediate from the definitions.
\\
\indent (6) Suppose that $\tau$ is refinable and $R$ is a $\tau$-$\beta$-WFFR.  Then from Theorem \ref{thm: usual diagram}, we know that $R$ is a $\tau$-$\alpha'$-divisor finite ring.  Furthermore, from \cite{Mooney}, we know that a $\tau$-$\beta$-WFFR with a refinable $\tau$ is $\tau$-$\alpha'$.  Because $\tau$ is refinable, if $R$ is $\tau$-$\alpha'$, then $R$ is $\tau$-$\alpha$.
\end{proof}

\begin{theorem} \label{thm: complete bfr} Let $R$ be a commutative ring with $1$ and $\tau$ be a symmetric relation on $R^{\#}$.  We have the following.
\\
(1) If $R$ is a BFR, then $R$ is a $\tau$-BFR.
\\
(2) If $R$ is a $\tau$-BFR, then $R$ is a $\tau$-complete-BFR.
\\
(3) Let $R$ be $\tau$-complete and $\tau$ refinable.  Then $R$ is a $\tau$-complete-BFR implies $R$ is a $\tau$-BFR.
\\
(4) Let $R$ be $\tau$-completable.  Then $R$ is a $\tau$-complete-BFR implies $R$ is a $\tau$-BFR.
\end{theorem}
\begin{proof}
(1) Let $R$ be a BFR, and $a$ be a non-unit.  Suppose $N(a)$ is the bound on the length of any factorization of $a$.  Any $\tau$-factorization $a=\lambda a_1 \cdots a_n$ is certainly a factorization, so $n\leq N(a)$, proving $R$ is a $\tau$-BFR.
\\
\indent (2) Let $R$ be a $\tau$-BFR, and $a$ be a non-unit.  Suppose $N(a)$ is the bound on the length of any $\tau$-factorization of $a$.  Any $\tau$-complete-factorization $a=\lambda a_1 \cdots a_n$ is certainly a $\tau$-factorization, so $n\leq N(a)$, proving $R$ is a $\tau$-complete-BFR.
\\
\indent (3) Let $R$ be $\tau$-complete and $\tau$ refinable.  Suppose $R$ is a $\tau$-complete-BFR.  Let $a$ be a non-unit.  Let $N(a)$ be the bound on the length of any $\tau$-complete factorization.  We claim this also serves as a bound on the length of any $\tau$-factorization. Let $a=\lambda a_1 \cdots a_n$ be any $\tau$-factorization of $a$.  Because $R$ is $\tau$-complete, each $a_i$ has a $\tau$-complete factorization and $\tau$ is refinable, we can $\tau$-refine this factorization into a $\tau$-complete factorization, say $a=\lambda' b_1 \cdots b_m$.  We have $n \leq m \leq N(a)$ as desired.
\\
\indent (4) Let $R$ be $\tau$-completable.  Suppose $R$ is a $\tau$-complete-BFR.  Let $a$ be a non-unit.  Let $N(a)$ be the bound on the length of any $\tau$-complete factorization.  This also serves as a bound on the length of any $\tau$-factorization. Let $a=\lambda a_1 \cdots a_n$ be any $\tau$-factorization of $a$.  By hypothesis, we can $\tau$-refine this factorization into a $\tau$-complete factorization, say $a=\lambda' b_1 \cdots b_m$.  We have $n \leq m \leq N(a)$ as desired.
\end{proof}
\begin{theorem} \label{thm: complete ffr}Let $R$ be a commutative ring with $1$ and $\tau$ be a symmetric relation on $R^{\#}$.  Let $\beta \in \{ $associate, strongly associate, very strongly associate $\}$.  We have the following.
\\
(1) If $R$ is a $\beta$-FFR, then $R$ is a $\tau$-$\beta$-FFR.
\\
(2) If $R$ is a $\tau$-$\beta$-FFR, then $R$ is a $\tau$-complete-$\beta$-FFR.
\\
(3) Let $R$ be $\tau$-complete and $\tau$ refinable.  Then $R$ is a $\tau$-complete-FFR implies $R$ is a $\tau$-$\beta$-FFR.
\\
(4) Let $R$ be $\tau$-completable.  Then $R$ is a $\tau$-complete-FFR implies $R$ is a $\tau$-$\beta$-FFR.
\end{theorem}
\begin{proof}
(1) Let $a$ be a non-unit.  The set of $\tau$-factorizations of $a$ up to $\beta$ is among the set of factorizations of $a$.  By hypothesis the latter is finite, so certainly the former is.
\\
\indent (2) Let $a$ be a non-unit.  The set of $\tau$-complete-factorizations of $a$ up to $\beta$ is among the set of $\tau$-factorizations of $a$.  By hypothesis, the latter is finite, so certainly the former is.
\\
\indent (3) Let $a$ be a non-unit.  We claim every $\tau$-factorization of $a$, up to $\beta$ can be realized as coming from grouping of factors of a $\tau$-complete-factorization up to $\beta$.  Since there are a finite number of $\tau$-factorizations up to $\beta$, each with a finite number of factors, there is a finite number of ways of grouping the factors to generate different $\tau$-factorizations up to $\beta$.  Given a $\tau$-factorization $a=\lambda a_1 \cdots a_n$, there is a $\tau$-complete factorization for each $a_i$ with $1 \leq i \leq n$.  Suppose $a_i = \lambda_i b_{i1} \cdots {im_i}$ is the $\tau$-complete factorization of $a_i$ for $1 \leq i \leq n$.  By hypothesis, $\tau$ is refinable, so 
\begin{equation} \label{eq:complete} a=(\lambda \lambda_1 \cdots \lambda_n)b_{11}\cdots b_{1m_1}\cdot b_{21}\cdots b_{2m_2} \cdots b_{n1} \cdots b_{nm_n}
\end{equation}
is a $\tau$-factorization.  By Theorem \ref{lem: complete refinement}, this is a $\tau$-complete factorization and hence was among the finite number of $\tau$-complete factorizations of $a$ up to $\beta$.
\\
\indent (4) This proof is nearly identical to the proof of (3).  The only modification is that we can use the fact that since $R$ is $\tau$-completable to automatically conclude that any factorization $a=\lambda a_1 \cdots a_n$ can be $\tau$-refined into a $\tau$-complete factorization of the form of Equation \ref{eq:complete}.  
\end{proof}
\begin{theorem} \label{thm: cdf-ring}Let $R$ be a commutative ring with $1$ and $\tau$ be a symmetric relation on $R^{\#}$.  Let $\beta \in \{ $associate, strongly associate, very strongly associate $\}$.  We have the following.
\\
(1) If $R$ is a $\beta$-WFFR, then $R$ is a $\tau$-$\beta$-WFFR.
\\
(2) If $R$ is a $\tau$-$\beta$-WFFR, then $R$ is a $\tau$-complete-$\beta$-divisor finite ring.
\\
(3) If $R$ is a $\tau$-$\beta$-atomic divisor finite ring, then $R$ is a $\tau$-complete-$\beta$-divisor finite ring.
\\
(4) If $R$ is a $\tau$-$\beta$-strongly atomic divisor finite ring, then $R$ is a $\tau$-complete-$\beta$-divisor finite ring.
\\
(5) If $\tau$ is refinable and $R$ is a $\tau$-$\beta$-m-atomic divisor finite ring, then $R$ is a $\tau$-complete-$\beta$-divisor finite ring.
\end{theorem}
\begin{proof} (1) Let $a\in R$ be a non-unit. If $a$ has a finite number of divisors up to $\beta$, then it certainly has a finite number  $\tau$-divisors up to $\beta$.
\\
\indent (2) Let $a\in R$ be a non-unit.  If there are a finite number of $\tau$-divisors of $a$ up to $\beta$, then certainly there are a finite number of $\tau$-divisors which occur as a $\tau$-factor in some $\tau$-complete-factorization of $a$ up to $\beta$.
\\
\indent (3) (resp. (4)) Let $a\in R$ be a non-unit.  Suppose $\{a_i\}_{i=1}^{\infty}$ is a infinite collection of non-$\beta$ $\tau$-divisors which occur in some $\tau$-complete factorization of $a$.  Say $a=\lambda_i a_i b_{i1} \cdots b_{in_i}$ is one such $\tau$-complete factorization.  By Theorem \ref{thm: complete-atoms}, this $\tau$-complete factorization is $\tau$-atomic (resp. strongly atomic).  This provides an infinite number of non-$\beta$ $\tau$-atomic (resp. $\tau$-strongly atomic) divisors of $a$, a contradiction.
\\
\indent (5) Let $a\in R$ be a non-unit.  Suppose $\{a_i\}_{i=1}^{\infty}$ is a infinite collection of non-$\beta$ $\tau$-divisors which occur in some $\tau$-complete factorization of $a$.  Say $a=\lambda_i a_i b_{i1} \cdots b_{in_i}$ is one such $\tau$-complete factorization.  We have a $\tau$ which is refinable, so by Theorem \ref{thm: complete-atoms}, this $\tau$-complete factorization is $\tau$-m-atomic.  This provides an infinite number of non-$\beta$ $\tau$-m-atomic divisors of $a$, a contradiction.
\end{proof}
We notice at this point that many of the $\tau$-finite factorization and $\tau$-complete finite factorization properties result in $R$ having the property that for a given non-unit $a\in R$, there is a finite number of divisors of $a$ which occur as a $\tau$-factor of some $\tau$-complete factorization.  We summarize these in the form of the following corollary.
\begin{corollary} Let $R$ be a commutative ring with $1$ and $\tau$ be a symmetric relation on $R^{\#}$.  Let $\beta \in \{ $associate, strongly associate, very strongly associate $\}$.  If $R$ satisfies any of the following conditions, then $R$ is a $\tau$-complete-$\beta$-divisor finite ring.
\\
(1) $R$ is a $\beta$-FFR.
\\
(2) $R$ is a $\tau$-$\beta$-FFR.
\\
(3) $R$ is a $\tau$-complete-$\beta$-FFR.
\\
(4) $R$ is a $\tau$-complete (completable)-$\beta$-UFR. 
\\
(5) $R$ is a $\beta$-WFFR.
\\
(6) $R$ is a $\tau$-$\beta$-WFFR.
\\
(7) $R$ is a $\tau$-$\beta$-irreducible (resp. strongly irreducible) divisor finite ring.
\\
(8) $R$ is a strongly associate ring and $R$ is a $\tau$-$\beta$-m-irreducible divisor finite ring.
\end{corollary} 
\begin{proof}We have seen in Theorem \ref{thm: main diagram} part (2) that (4) $\Rightarrow$ (1).  By Theorem \ref{thm: complete ffr}, we have (1) $\Rightarrow$ (2) $\Rightarrow$ (3) and Theorem \ref{thm: main diagram} part (5) proves that (3) implies $R$ is a $\tau$-$\beta$-cdf ring.  
\\
\indent We know from Theorem \ref{thm: cdf-ring} part (1) and (2), that (5) $\Rightarrow$ (6) and that $(6)$ implies $R$ is a $\tau$-$\beta$-cdf ring.  (7) and (8) are restatements of \ref{thm: cdf-ring} parts (3), (4) and (5).  This completes the proof.
\end{proof}

The following diagram serves as an illustration which attempts to combine several of the previous results regarding various $\tau$-complete finite factorization properties.  Let $\approx$ represent a strongly associate ring, $\nabla$ represent $\tau$ is refinable and let $\dagger$ represent $R$ is both $\tau$-complete and $\tau$-is refinable.  Let $\gamma \in \{$ complete, completable, atomic, atomicable, strongly atomic, strongly atomicable, m-atomic, m-atomicable, very strongly atomic, very strongly atomicable $\}$.
\small
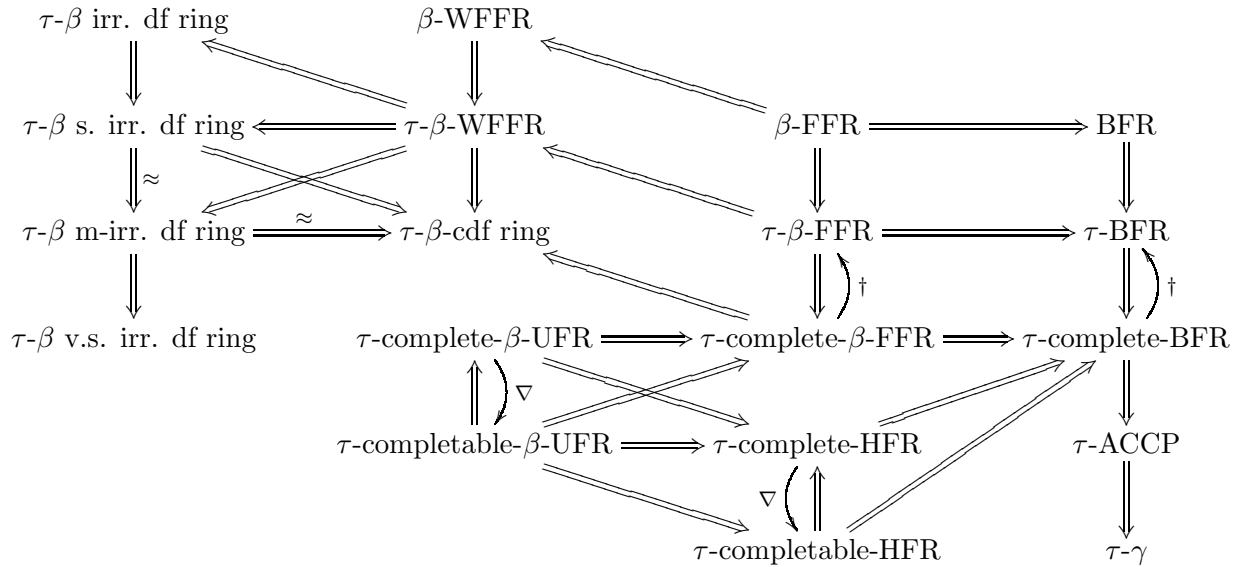
{$$\xymatrix{
\tau\text{-}\beta \text{ irr. df ring} \ar@{=>}[d]&\beta\text{-WFFR}\ar@{=>}[d]       &         &        \\
\tau \text{-} \beta \text{ s. irr. df ring} \ar@{=>}[dr]\ar@{=>}^{\approx}[d]&\tau\text{-}\beta\text{-WFFR} \ar@{=>}[d]\ar@{=>}[l]\ar@{=>}[ul]\ar@{=>}[dl]& \beta \text{-FFR} \ar@{=>}[ul]\ar@{=>}[r]\ar@{=>}[d] & \text{BFR} \ar@{=>}[d] \\
\tau \text{-} \beta \text{ m-irr. df ring} \ar@{=>}[d]\ar@{=>}^{\approx}[r]&\tau\text{-}\beta\text{-cdf ring} & \tau \text{-}\beta \text{-FFR} \ar@{=>}[r]\ar@{=>}[ul] \ar@{=>}[d] & \tau\text{-BFR} \ar@{=>}[d] \\
\tau\text{-}\beta \text{ v.s. irr. df ring}&\tau\text{-complete-}\beta \text{-UFR}\ar@{=>}[r] \ar@/^1pc/^{\nabla}[d] \ar@{=>}[dr] & \tau \text{-complete-}\beta \text{-FFR} \ar@{=>}[ul]\ar@{=>}[r]\ar@/_1pc/_{\dagger}[u]&\tau \text{-complete-BFR}\ar@/_1pc/_{\dagger}[u] \ar@{=>}[d]\\
&\tau\text{-completable-}\beta \text{-UFR} \ar@{=>}[u] \ar@{=>}[r] \ar@{=>}[ur] \ar@{=>}[dr]& \tau \text{-complete-HFR} \ar@/_1pc/_{\nabla}[d]\ar@{=>}[ur]&  \tau\text{-ACCP} \ar@{=>}[d]\\
&  & \tau \text{-completable-HFR} \ar@{=>}[u] \ar@{=>}[uur]& \tau\text{-}\gamma }$$}\\
\section*{Acknowledgment}
The author would like to acknowledge The University of Iowa and Professor Daniel D. Anderson for their support while much of the research appearing in this article was completed as a Presidential Graduate Research Fellow.

\bibliographystyle{plain}
\bibliography{bibliography}

\begin{thebibliography}{10}

\bibitem{Stickles}
D.D. Anderson, M.~Axtell, S.~J. Forman, and J.~Stickles.
\newblock When are associates unit multiples?
\newblock {\em Rocky Mountain J. Math.}, 34(3):811--828, 2004.

\bibitem{Frazier}
D.D. Anderson and A.~Frazier.
\newblock On a general theory of factorization in integral domains.
\newblock {\em Rocky Mountain J. Math.}, 41:3:663--705, 2011.

\bibitem{andersonzdg}
D.D. Anderson and M.~Naseer.
\newblock Beck's coloring of a commutative ring.
\newblock {\em J. Algebra}, 159:2:500--514, 1993.

\bibitem{Valdezleon}
D.D. Anderson and S.~Valdes-Leon.
\newblock Factorization in commutative rings with zero divisors.
\newblock {\em Rocky Mountain J. Math.}, 26:2:439--480, 1996.

\bibitem{Axtellzdg}
D.F. Anderson, Michael~C. Axtell, and Joe~A. Stickles, Jr.
\newblock Zero-divisor graphs in commutative rings.
\newblock In {\em Commutative algebra---{N}oetherian and non-{N}oetherian
  perspectives}, pages 23--45. Springer, New York, 2011.

\bibitem{davidanderson}
D.F. Anderson, A.~Frazier, A.~Lauve, and P.S. Livingston.
\newblock The zero-divisor graph of a commutative ring. {II}.
\newblock In {\em Ideal theoretic methods in commutative algebra ({C}olumbia,
  {MO}, 1999)}, volume 220 of {\em Lecture Notes in Pure and Appl. Math.},
  pages 61--72. Dekker, New York, 2001.

\bibitem{Livingston}
D.F. Anderson and P.S. Livingston.
\newblock The zero-divisor graph of a commutative ring.
\newblock {\em J. Algebra}, 217:2:434--447, 1999.

\bibitem{Axtell}
M.~Axtell.
\newblock U-factorizations in commutative rings with zero-divisors.
\newblock {\em Comm. Algebra}, 30:3:1241--1255, 2002.

\bibitem{Axtell2}
M.~Axtell, S.~Forman, N.~Roersma, and J.~Stickles.
\newblock Properties of u-factorizations.
\newblock {\em International Journal of Commutative Rings}, 2:2:83--99, 2003.

\bibitem{Beck}
I.~Beck.
\newblock Coloring of commutative rings.
\newblock {\em J. Algebra}, 116:208--226, 1988.

\bibitem{bouvier71}
A.~Bouvier.
\newblock Sur les anneaux de fractions des anneaux atomiques
  pr\'esimpliﬁables.
\newblock {\em Bull. Sci. Math.}, 95:371--376, 1971.

\bibitem{bouvier72a}
A.~Bouvier.
\newblock Anneaux pr\'esimpliﬁables.
\newblock {\em C. R. Acad. Sci. Paris S\'er. A-B}, 274:1605--1607, 1972.

\bibitem{bouvier72b}
A.~Bouvier.
\newblock R\'esultats nouveaux sur les anneaux pr\'esimpliﬁables.
\newblock {\em C. R. Acad. Sci. Paris S\'er. A-B}, 275:955--957, 1972.

\bibitem{bouvier74}
A.~Bouvier.
\newblock Anneaux pr\'esimpliﬁables.
\newblock {\em Rev. Roumaine Math. Pures Appl.}, 19:713--724, 1974.

\bibitem{Fletcher}
C.R. Fletcher.
\newblock Unique factorization rings.
\newblock {\em Proc. Cambridge Philos. Soc.}, 65:579--583, 1969.

\bibitem{Fletcher2}
C.R. Fletcher.
\newblock The structure of unique factorization rings.
\newblock {\em Proc. Cambridge Philos. Soc.}, 67:535--540, 1970.

\bibitem{Juettcomax}
Jason Juett.
\newblock Generalized comaximal factorization of ideals.
\newblock {\em J. Algebra}, 352:141--166, 2012.

\bibitem{Mcadam}
S.~McAdam and R.~Swan.
\newblock Unique comaximal factorization.
\newblock {\em J. Algebra}, 276:180--192, 2004.

\bibitem{Mooney}
C.P. Mooney.
\newblock Generalized factorization in commutative rings with zero-divisors.
\newblock {\em Houston J. Math.}, to appear.

\bibitem{Mooney2}
C.P. Mooney.
\newblock Generalized u-factorization in commutative rings with zero-divisors.
\newblock {\em Rocky Mountain J. Math.}, to appear.

\end{thebibliography}

\end{document}